\documentclass{article}
\usepackage{amsmath,amsthm,amssymb,graphicx}
\usepackage{tikz}

\newtheorem{theorem}{Theorem}[section]
\newtheorem{lemma}[theorem]{Lemma}
\newtheorem{proposition}[theorem]{Proposition}

\newtheorem{definition}[theorem]{Definition}

\theoremstyle{remark}
\newtheorem{remark}[theorem]{Remark}
\newtheorem{example}[theorem]{Example}

\begin{document}

\title{Explicit Boij--S\"oderberg theory of ideals\\ from a graph isomorphism reduction}

\author{Alexander Engstr\"om \and Laura Jakobsson \and Milo Orlich\footnote{{Department of Mathematics and Systems Analysis, Aalto University, Espoo, Finland \newline \tt \{alexander.engstrom, laura.p.jakobsson, milo.orlich\} @aalto.fi}} }

\date\today

\maketitle

\begin{abstract}
In the origins of complexity theory Booth and Lueker showed that the question of whether two graphs are isomorphic or not can be reduced to the special case of chordal graphs. To prove that, they defined a transformation from graphs $G$ to chordal graphs $BL(G)$. The projective resolutions of the associated edge ideals $I_{BL(G)}$ is manageable and we investigate to what extent their Betti tables also tell non-isomorphic graphs apart. It turns out that the coefficients describing the decompositions of Betti tables into pure diagrams in Boij--S\"oderberg theory are much more explicit than the Betti tables themselves, and they are expressed in terms of classical statistics of the graph $G.$
\end{abstract}

\section{Introduction}

According to the main theorem of Boij--S\"oderberg theory every Betti table can be expressed as a weighted non-negative sum of particularly elementary tables, called pure Betti tables \cite{BS08,ES11,Fl}. These weight coefficients are usually even more cumbersome to express than the Betti numbers themselves, and from explicit calculations they tend to involve many binomial coefficients and alternating signs. In our setting it turns out that they are more straightforward to state than the Betti numbers, and we can provide elementary formulas for them.

In one of the first results of complexity theory, Booth and Lueker \cite{BL75} proved that two graphs $G$ and $G'$ are isomorphic if and only if their corresponding Booth--Lueker graphs $BL(G)$ and $BL(G')$ are isomorphic. This reduced the problem of graph isomorphism to the special class of Booth--Lueker graphs, which have several attractive properties. The main result of this paper, Theorem~\ref{thm:main}, regards the edge ideal of the Booth--Lueker graph of a graph on $n$ vertices and $d_i$ vertices of degree $i$, for $i=0,\dots,n-1$. If $c_j$ is the weight of the pure 2-linear table with $j$ non-zero entries on the second row, then
\[
c_j =  
\frac{d_{j-n+1}}{j}+\frac{1}{j(j+1)} \sum_{i=j-n}^{n-1}d_i
\,\,\,\textrm{ for }\,\,\, n \leq j \leq 2n-2,\,\,\,\,\,\,\,\,\,
c_{n-1}=\frac{d_0}{n},
\]
and the other Boij--S\"oderberg coefficients vanish. Further on, we also describe in Section~\ref{sec:main} the explicit Betti numbers and anti-lecture hall compositions of these ideals.

In Theorem~\ref{thm:mainComp} we study the dual graph situation, constructing the ideal of the complement of the Booth--Lueker graph. We show that if the original graph has $n$ vertices and $m$ edges, then
\[ c_j= \frac m{j(j+1)}\,\,\,\textrm{ for }\,\,\,m\le j\le m+n-4,\,\,\,\,\,\,\,\,\,
c_{m+n-3}=\frac{m}{m+n-3},
\]
and the other coefficients vanish. We also give explicit results on Betti numbers and anti-lecture hall compositions for these ideals in Section~\ref{sec:mainComp}.

\subsubsection*{Acknowledgements} We would like to thank Petteri Kaski for discussions on the graph isomorphism problem.

\section{Tools and background results}

In this section we list some useful tools from graph theory, commutative algebra and combinatorics, without introducing anything new. 

\subsection{Basic definitions and standard tools in graph theory and combinatorics}


All graphs are assumed to be simple, meaning that the edges have no direction and there are no multiple edges or loops. The \emph{degree} of a vertex is the number of edges incident to it. It is common to discuss the degree sequence of a graph, that is, the degrees sorted downwards. We will \emph{not} do that. 
\begin{definition}
The \emph{degree vector} or \emph{degree statistics} of a graph $G$ on $n$ vertices is the column vector
\[ \mathbf{d}_ G:= (d_0, d_1, \ldots, d_{n-1} )^T \]
where $d_i$ is the number of vertices of degree $i$ in $G.$
\end{definition}
\begin{remark}
Actually the word \emph{degree} was brought from algebra to graph theory by Petersen~\cite{P91} in 1891, addressing a problem in invariant theory by Hilbert.
\end{remark}

The complement of a graph $G$ is denoted by $\overline{G}$ and the induced subgraph of $G$ on the set of vertices $W$ by $G[W].$ 


\begin{definition}
A graph $G$ is said to be \emph{chordal} if every cycle of lenght greater than three has a chord.
\end{definition}

\begin{definition}
Given a graph $G$, we denote by $I_G:=(x_ix_j\mid ij\in E(G))$ its \emph{edge ideal}, in a polynomial ring $S:=k[x_1,\dots,x_n]$ with as many variables as the vertices of $G$, where $k$ is a field.
\end{definition}

\begin{remark}
We are interested in the Betti numbers of the edge ideals of some chordal graphs. Corollary 5.10 of \cite{HV} shows that the characteristic of the field $k$ does not affect these Betti numbers.
\end{remark}


The following well-known result is used several times later on, and proved in for example Section~1.2.6.I of~\cite{K73}. The first equality is known as Vandermonde's identity or convolution.

\begin{lemma}\label{littlebinomialformulas}
For all integers $r$, $s$, and $n$, 
\[ \sum_{k}\binom rk\binom s{n-k}=\binom{r+s}n \,\,\,\textrm{ and }\,\,\, \sum_{k}(-1)^{r-k}\binom rk\binom{s+k}n=\binom s{n-r}. \]
\end{lemma}

We will also make use several times of the following combinatorial result. It is also well-known and appears in many sources. For lack of a reference with a complete proof, we include our own.

\begin{lemma}\label{lem:thesumiszero}
Let $P\in \mathbb{Q}[x]$ be a polynomial of degree less than $N$ such that $P(n)\in\mathbb{Z}$ for all $n\in\mathbb{Z}$. Then
\begin{equation}\label{bigzerosum}
\sum_{i=0}^N(-1)^i\binom NiP(i)=0.
\end{equation}
\end{lemma}
\begin{proof}
By Lemma~4.1.4 in~\cite{BH93}, there are $a_0,\dots,a_{N-1}$ in $\mathbb{Z}$ such that
$$P(x)=\sum_{j=0}^{N-1}a_j\binom{x+j}j.$$ Therefore, it is enough to prove that the statement holds for any polynomial of the form $\binom{x+j}j$, with $j<N$. This follows from differentiating $j$ times the binomial formula $(1+x)^N=\sum_{i=0}^N\binom Nix^i$, and setting $x=-1$.
\end{proof}

\subsection{Booth--Lueker graphs and the graph isomorphism problem}

In this section we introduce our object of study, namely that of the Booth--Luker graph $BL(G)$ associated to a given graph $G$, but before doing so we explain the motivation behind the construction.

A fundamental question in graph theory (with applications primarly in computer science) that goes by the name of ``graph isomorphism problem'' is the following: are two given graphs $G$ and $G'$ isomorphic? It is still an open question whether the isomorphism for general graphs can be determined in polynomial time or not. It is also unknown whether the graph isomorphism problem is NP-complete.  There are some upper bounds: in 1983 Babai and Luks~\cite{BL83} showed that the problem can be solved in moderately exponential $\exp(O(\sqrt{n\log n}))$ time, where $n$ is the number of vertices, and in 2015 Babai claimed that the problem can actually be solved in quasipolynomial time. A mistake in the proof was found by Helfgott and in 2017 Babai re-claimed the quasipolynomial time (see~\cite{Ba}) after dealing with the error. As far as we know, the proof has not yet been fully checked.

However, there are many results known about  special classes of graphs. We are interested in the case of chordal graphs, due to the following polynomial-time reduction result proved by Booth and Lueker.
\begin{theorem}[\cite{BL75}, Theorem 4.6]
Arbitrary graph isomorphism is polynomially reducible to chordal graph isomorphism.
\end{theorem}
To prove this, they made use of the following construction, which we therefore name after them.

\begin{definition}
For any graph $G$ let $BL(G)$ be the graph with vertex set $V(G) \cup E(G)$ and edges $uv$ for every pair of vertices in $G$ and $ue$ for every vertex $u$ incident to an edge $e$ in $G$. We call $BL(G)$ the \emph{Booth--Lueker graph} of $G$.
\end{definition}

Observe that both $BL(G)$ and its complement are chordal, for every $G$. They are split graphs. Thus, we actually have two interesting ideals to define. In Section~\ref{sec:main} we study algebraic invariants of the edge ideals of Booth--Lueker graphs, and in Section~\ref{sec:mainComp} we do the same for the complements.

\begin{remark}
As already stated, for two graphs $G$ and $G'$, having an isomorphism $G\cong G'$ is equivalent to having $BL(G)\cong BL(G')$. So the Booth--Lueker construction provides a reduction of the isomorphism problem from general graphs to the class of chordal graphs, which are particularly tractable thanks to Theorem~\ref{bettibyconncomp}.
\end{remark}

\begin{example}
Consider the path on four vertices. It is depicted in Figure~\ref{fig:BL} with its Booth-Lueker graph.
\end{example}

\begin{figure}
\begin{center}
\begin{tikzpicture} [>=latex]
\fill (-4,0) circle (0.1);
\fill (-4,1) circle (0.1);
\fill (-3,0) circle (0.1);
\fill (-3,1) circle (0.1);
\draw [thick] (-3,1) -- (-3,0);
\draw [thick] (-3,1) -- (-4,1);
\draw [thick] (-4,0) -- (-3,0);
\fill (0,0) circle (0.1);
\fill (0,1) circle (0.1);
\fill (1,0) circle (0.1);
\fill (1,1) circle (0.1);
\fill (2,1.5) circle (0.1);
\fill (2,.5) circle (0.1);
\fill (2,-.5) circle (0.1);
\draw [thick] (0,0) -- (1,0);
\draw [thick] (1,1) -- (0,1);
\draw [thick] (1,1) -- (1,0);
\draw [thick] (0,0) -- (0,1);
\draw [thick] (0,0) -- (1,1);
\draw [thick] (1,0) -- (0,1);
\draw [thick] (2,1.5) -- (1,1);
\draw [thick] (2,1.5) -- (0,1);
\draw [thick] (2,.5) -- (1,1);
\draw [thick] (2,.5) -- (1,0);
\draw [thick] (2,-.5) -- (1,0);
\draw [thick] (2,-.5) -- (0,0);
\coordinate [label=$\mapsto$] (B) at (-1.5,.25);
\coordinate [label=$G$] (G) at (-3.5,-1.03);
\coordinate [label=$BL(G)$] (BLG) at (.5,-1.1);
\end{tikzpicture}
\caption{The path on four vertices and its Booth--Lueker graph.}
\label{fig:BL}
\end{center}
\end{figure}
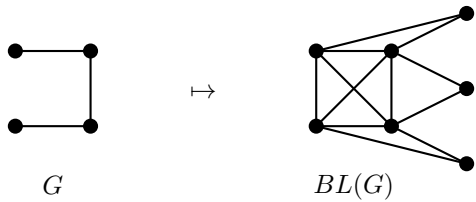

We can think of the Booth--Lueker constructions in the following terms: we take a copy of the original graph $G$ on the left, add on the right new vertices corresponding to the edges of $G$ and connect these vertices with the respective ends of the associated edge of $G$. Then we complete the left part of the graph. Later on we will use again this terminology of the ``left'' and ``right'' part of $BL(G)$.

\subsection{Boij--S\"oderberg theory}

Denote as usual $S=k[x_1,\dots,x_n]$, where $k$ is a field. For a finitely generated graded $S$-module $M$, the $(i,j)$th \emph{(graded) Betti number} is defined as
$$\beta_{i,j}(M):=\dim_k(\mathrm{Tor}^S_i(k,M)_j).$$
It is customary to arrange the Betti numbers of $M$ in the \emph{Betti table} of $M$, which has as $ji$th entry the number $\beta_{i,i+j}(M)$.

Boij--S\"oderberg theory deals with writing the Betti table of a finitely generated graded $S$-module as a sum of simpler pieces, coming from the so-called ``pure Betti tables'': to each sequence $\mathbf n=(n_0,\dots,n_s)$ of strictly increasing non-negative integers, we associate the table $\pi(\mathbf n)$ with entries
$$\pi(\mathbf n)_{i,j}:=
\begin{cases}
\prod_{k\ne0,i}\vert\frac{n_k-n_0}{n_k-n_i}\vert & \text{if }i\ge0,j=n_i,\\
0&\text{otherwise.}
\end{cases}$$
This is called the \emph{pure Betti table} associated to $\mathbf n$. We can moreover give a partial order to such sequences by setting
$$(n_0,\dots,n_s)\ge (m_0,\dots,m_t)$$
whenever $s\le t$ and $n_i\ge m_i$ for all $i\in\{0,\dots,s\}$. 

\begin{remark}
It is customary to refer to a sequence of stricly increasing integers as above as a \emph{degree sequence}. Since the same name also corresponds to a concept in graph theory, we just do not refer to them in order to avoid confusion.
\end{remark}

\begin{theorem}[\cite{Fl}, Theorem 5.1]
For every finitely generated graded $S$-module $M$, there is a stricly increasing chain $\mathbf n_1<\dots<\mathbf n_p$ of stricly increasing sequences of $n+1$ non-negative integers and there are numbers $c_{\mathbf n_1},\dots,c_{\mathbf n_p}\in\mathbb Q_{\ge0}$ such that
$$\beta(M)=c_{\mathbf n_1}\pi(\mathbf n_1)+\dots+c_{\mathbf n_p}\pi(\mathbf n_p).$$
\end{theorem}

\begin{definition}
We refer to  non-negative rational numbers $c_{\mathbf n_1},\dots,c_{\mathbf n_p}$ as in the theorem above as \emph{Boij--S\"oderberg coefficients} of $M$.
\end{definition}

\begin{example}
Let  $I=(x^2,xy,y^3)\subset S=k[x,y]$. Then one can compute that
\begin{align*}
\beta(S/I)&=\left(\begin{array}{ccc}
1&0&0\\
0&2&1\\
0&1&1
\end{array}\right)\\
&=\frac12\left(\begin{array}{ccc}
1&0&0\\
0&3&2\\
0&0&0
\end{array}\right)
+\frac14\left(\begin{array}{ccc}
1&0&0\\
0&2&0\\
0&0&1
\end{array}\right)
+\frac14\left(\begin{array}{ccc}
1&0&0\\
0&0&0\\
0&4&3
\end{array}\right)\\
&=\frac12\pi(0,2,3)+\frac14\pi(0,2,4)+\frac14\pi(0,3,4).
\end{align*}
Hence $\mathbf n_1=(0,2,3)$, $\mathbf n_2=(0,2,4)$, $\mathbf n_3=(0,3,4)$, and the Boij--S\"oderberg coefficients are $c_{\mathbf n_1}=1/2$, $c_{\mathbf n_2}=1/4$, and $c_{\mathbf n_3}=1/4$.
\end{example}

\begin{remark}
In the same way as for the graded Betti numbers, there is a very straightforward algorithm to compute Boij--S\"oderberg coefficients, but there are no general explicit formulas.
\end{remark}

For a more detailed account of Boij--S\"oderberg theory, see for instance the survey~\cite{Fl} by Fl\o ystad. There the notation for a degree sequence is $\mathbf d$ instead of $\mathbf n$, but we already use $\mathbf d_G$ to denote the degree vector of a graph $G$.

As motivated in Section~\ref{sec:linres}, we will mostly be interested in ideals $I$ with \emph{$2$-linear resolutions}, i.e., such that the Betti table of $S/I$ is of the form
$$\beta(S/I)=\left(
\begin{array}{ccccc}
1 & 0 & 0 & \cdots &0\\
0 & \beta_{1,2} & \beta_{2,3} & \cdots & \beta_{p,p+1}\\
\end{array}
\right).$$ By Boij--S\"oderberg theory, such a Betti table will be the weighted average of certain pure tables of the form $\pi(0,2,3,\dots,s,s+1)$. For instance
$$\pi(0,2)=\left(
\begin{array}{cccc}
1 & 0 & 0 & \cdots \\
0 & 1 & 0 & \cdots \\
\end{array}
\right)$$
or
$$\pi(0,2,3)=\left(
\begin{array}{ccccc}
1 & 0 & 0 & 0 & \cdots \\
0 & 3 & 2 & 0 & \cdots \\
\end{array}
\right).$$
Following the discussion around converting Betti diagrams and Boij--S\"oderberg coefficients into each other from the paper~\cite{ES13}, we record these pure diagrams as the two vectors
$$\pi_1=( 1,0 ,\dots )\qquad\text{and}\qquad\pi_2=(3, 2, 0 \dots ),$$
i.e., we denote by $\pi_s$ the second row of $\pi(0,2,3,\dots,s,s+1)$, except for the first entry, which is always zero.
Slightly larger examples are
\[
\begin{array}{rcl}
 \pi_7 & = & (28 , 112 , 210, 224, 140, 48, 7, 0, 0, \dots ),  \\
 \pi_8 & =  & (36, 168, 378, 504, 420, 216, 63, 8, 0, 0, \dots ), \\
 \pi_9 & =  & (45, 240, 630, 1008, 1050, 720, 315, 80, 9, 0, 0, \dots ). 
 \end{array}
 \]
The zeros in the end extend indefinitely, and it will be convenient for us to write a specific amount of them, as it will become clear in the next sections. We also denote from now on $c_s:=c_{\pi_s}$.

\subsection{On 2-linear resolutions}\label{sec:linres}

\begin{definition}
A \emph{$2$-linear resolution} of a graded $S$-module $M$ is a resolution where $\beta_{i,j}(M)=0$ if $j\ne i+1s$. We call an ideal with a $2$-linear resolution a \emph{$2$-linear ideal}.
\end{definition}

The following was proved by Fr\"oberg~\cite{F90} and refined by Dochtermann and Engstr\"om~\cite{DE09}.
\begin{theorem}[\cite{DE09}, Theorem 3.4]\label{bettibyconncomp}
For a simple graph $G$, the edge ideal $I_G$ has $2$-linear resolution if and only if the complement of $G$ is chordal. Moreover, given a graph $G$ whose complement is chordal, the Betti numbers of $S/I_G$ can be computed as follows:
\[
\beta_{i,i+1}(S/I_G) = \sum_{W\subseteq V(G) \atop \#W = i+1} \big( -1 +  \text{the number of connected components of $\overline{G}[W]$} \big).
\]
\end{theorem}

Engstr\"om and Stamps showed in~\cite{ES13} how Betti tables and Boij--S\"oderberg coefficients are related for 2-linear resolutions. 
We summarise this in the following lemma.
First of all, let us introduce the following notation: if a graded $S$-module $M$ has Betti table
$$\beta(M)=\left(\begin{array}{ccccc}
m & 0 & 0 & \cdots &0\\
0 & \beta_{1,2} & \beta_{2,3} & \cdots & \beta_{p,p+1}\\
\end{array}\right),$$
we denote by $\omega(M)=(\beta_{1,2},\dots,\beta_{p,p+1})$ the \emph{Betti vector} of $M$.

\begin{lemma}[\cite{ES13}, Lemma~3.1 and Theorem~3.2]\label{lem:BSfromES}
Let $\Omega$ be the square matrix  of size $n+m-1$ whose $ij$-entry is $j\binom{i+1}{j+1}$. Then $\Omega$ is invertible and the inverse $\Omega^{-1}$ has $ij$-entry $(-1)^{i-j}\frac1i\binom{i+1}{j+1}$.
Furthermore, if $M$ has a Betti table as above and if  $c=(c_1,\dots,c_{n+m-1})$ is the vector with the Boij--S\"oderberg coefficients of $M$, then we have 
\[
c=\omega(M)\Omega^{-1}.
\]
\end{lemma}

\subsubsection{Anti-lecture hall compositions}\label{sec:antilecintro}

\begin{definition}
A sequence of integers $\lambda=(\lambda_1,\lambda_2,\dots,\lambda_n)$ such that
$$t\ge\frac{\lambda_1}1\ge\frac{\lambda_2}2\ge\dots\ge\frac{\lambda_n}n\ge0$$
is called an \emph{anti-lecture hall composition} of length $n$ bounded above by $t$. 
\end{definition}

For further information about anti-lecture hall compositions, see~\cite{CS03}, where they were introduced, and also~\cite{S16}.

To a 2-linear ideal (equivalently, a graph with chordal complement) we can associate a unique anti-lecture hall composition with $t=1$ and $\lambda_1=1$, see Section~$4$ of~\cite{ES13}.

\begin{lemma}[\cite{ES13}, Proposition~4.11]\label{ESantilecfrombetti}
Let $I$ be a $2$-linear ideal, denote by $\lambda=(\lambda_1,\dots,\lambda_{m-1})$ the anti-lecture hall composition associated to $I$, $\omega(S/I)$ the Betti vector of $S/I$ and $\Psi$ the invertible $m\times m$ matrix with $ij$-entry $\Psi_{ij}=\binom{i-1}{j-1}$. Then we have $$\lambda=\omega(S/I)\Psi^{-1}.$$
\end{lemma}

\section{The Boij--S\"oderberg theory of ideals of Booth--Lueker graphs}\label{sec:main}

In this section we determine the Betti tables, the Boij--S\"oderberg coefficients, and the anti-lecture hall compositions of the edge ideals of Booth--Lueker graphs. It turns out that the information carried over to the algebraic setting from the graphs are all compiled in its degree vector. We will employ several results by Engstr\"om and Stamps~\cite{ES13} on 2-linear resolutions and Boij--S\"oderberg theory to reach these goals.

\subsection{Degree vector and Betti numbers}

For an $S$-module $S/I$ with $2$-linear minimal resolution of length $n$, we denote the non-trivial part of its Betti table as
$$\omega(S/I):=(\beta_{1,2},\beta_{2,3},\dots,\beta_{n,n+1})$$
and call it the \emph{reduced Betti vector} of $S/I$. If the ideal $I$ is the edge ideal of a graph $G$, we just write $\omega(G)$.

\begin{proposition}[From the degree vector to the Betti numbers]\label{niceidea}
Let $G$ be a graph on $n$ vertices and $m$ edges, let $A$ be the matrix of size $(n+m-1)\times n$ defined by
$ A_{ij}=\binom{j+n-2}i, $ and let $v$ be the column $(n+m-1)$-vector defined by $v_i = \binom n{i+1}$. Then
\begin{equation}\label{bellapensata}
\omega(BL(G))=A\mathbf{d}_G-v,
\end{equation}
where $\mathbf{d}_{G}= (d_0, d_1, \ldots, d_{n-1} )^T$ is the degree vector of $G$.
\end{proposition}

\begin{proof}
On the right-hand side we have $$(A\mathbf{d}_G-v)_i=\sum_{j=1}^n\binom{j+n-2}id_{j-1}-\binom n{i+1}.$$ We want to use the formula in Theorem~\ref{bettibyconncomp}. Let $W$ be a set of $i+1$ vertices. If all of $W$ is in the left (independent) part of $\overline{BL(G)}$, then the induced subgraph has $i+1$ connected components. If one of the vertices of $W$ is in the right part of $\overline{BL(G)}$, then this vertex is connected in $\overline{BL(G)}$ to $n-2$ of the $n$ vertices on the left. So actually if $W$ has some vertices in the right part there are not many possibilities for the number of connected components of the induced subgraph: they can either be one, two or three. By applying the formula we find that
\begin{align*}
\beta_{i,i+1} & =i\binom n{i+1}+2\binom{n-2}{i-2}m+2\binom{n-2}{i-1}m+\sum_{k=2}^i\sum_{j=1}^nd_{j-1}\binom{j-1}k\binom{n-1}{i-k}\\
&=i\binom n{i+1}+\sum_{k=1}^i\sum_{j=1}^nd_{j-1}\binom{j-1}k\binom{n-1}{i-k},
\end{align*}
where we used that $\sum_{j=1}^n(j-1)d_{j-1}=2m$. So we want to prove that
$$(i+1)\binom n{i+1}+\sum_{j=1}^nd_{j-1}\sum_{k=1}^i\binom{j-1}k\binom{n-1}{i-k}=\sum_{j=1}^n\binom{j+n-2}id_{j-1},$$
which can be rewritten as
$$\sum_{j=1}^nd_{j-1}\sum_{k=0}^i\binom{j-1}k\binom{n-1}{i-k}=\sum_{j=1}^n\binom{j+n-2}id_{j-1}$$
since $(i+1)\binom n{i+1}=n\binom{n-1}i$. We can now conclude due to the special case of Vandermonde's identity  (see Lemma~\ref{littlebinomialformulas})
$$\sum_{k=0}^i\binom{j-1}k\binom{n-1}{i-k}=\binom{j+n-2}i.$$
\end{proof}

\begin{figure}
\begin{center}
\includegraphics[width=3cm
]{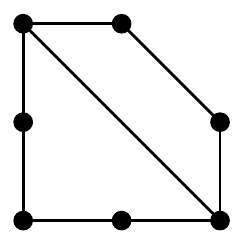}
\caption{The graph $G$ in Example~\ref{examplebn}.}
\label{fig:ge}
\end{center}
\end{figure}

\begin{example}\label{examplebn}
Consider the graph $G$ on 7 vertices and 8 edges depicted in Figure~\ref{fig:ge}. Due to Proposition~\ref{niceidea} we compute the Betti numbers of $BL(G)$. We consider the matrix
\[
A=
\left(
\begin{array}{ccccccc}
6 & 7 & 8 & 9 & 10 & 11 & 12 \\
15 & 21 & 28 & 36 & 45 & 55 & 66 \\
20 & 35 & 56 & 84 & 120 & 165 & 220 \\
15 & 35 & 70 & 126 & 210 & 330 & 495 \\
6 & 21 & 56 & 126 & 252 & 462 & 792 \\
1 & 7 & 28 & 84 & 210 & 462 & 924 \\
0 & 1 & 8 & 36 & 120 & 330 & 792 \\
0 & 0 & 1 & 9 & 45 & 165 & 495 \\
0 & 0 & 0 & 1 & 10 & 55 & 220 \\
0 & 0 & 0 & 0 & 1 & 11 & 66 \\
0 & 0 & 0 & 0 & 0 & 1 & 12 \\
0 & 0 & 0 & 0 & 0 & 0 & 1 \\
0 & 0 & 0 & 0 & 0 & 0 & 0 \\
0 & 0 & 0 & 0 & 0 & 0 & 0 \\ 
\end{array}
\right)
\]
constructed by $A_{ij} = {j+5 \choose i }$. The degree vector of $G$ is the column vector
\[ \mathbf{d}_{G} = ( 0,0,5,2,0,0,0)^T.\] Then
\[ A\mathbf{d}_{G} = 
(58, 212 , 448 , 602 , 532 , 308 , 112 , 23 , 2 , 0 , 0, 0, 0, 0)^T.
\]
Now let
\[ v = (
21 , 35 , 35 , 21 , 7 , 1 , 0 , 0 , 0 , 0 , 0 , 0 , 0, 0
)^T,
\]
where $v_i = {7 \choose i+1}$ for $i=1,2,\ldots, 14.$ Then the Betti vector is
\begin{align*}
\omega(BL(G))&=A\mathbf{d}_G-v\\
& = ( 37 , 177 , 413 , 581 , 525 , 307 , 112 , 23 , 2 , 0 , 0, 0 , 0 , 0)^T
\end{align*}
and the Betti table is
$$\beta(BL(G))=\left(
\begin{array}{ccccccccccccccc}
1&0&0&0&0&0&0&0&0&0&0&0&0&0&0\\
0&37&177&413&581&525&307&112&23&2&0&0&0&0&0
\end{array}\right).$$
\end{example}

\begin{remark}
The reason for the extra zeros comes from Hilbert's Syzygy Theorem, giving a bound on the projective dimension: we simply write all the possibly non-zero entries.
\end{remark}

\begin{proposition}[From the Betti numbers to the degree vector]\label{prop:bntodv}
Let $\Delta(G)$ be the largest vertex degree in $G$. Let $A$ be as in Proposition \ref{niceidea} and let $B$ be the square submatrix of $A$ obtained by taking the first $\Delta(G)+1$ columns and the rows from $n-1$ to $n+\Delta(G)-1$. Then we have $(B^{-1})_{ij}=(-1)^{i+j}B_{ij}$ and
\[
\mathbf{d}_{G}=B^{-1}(\beta_{n-1,n}+1,\beta_{n,n+1},\dots,\beta_{n+\Delta(G)-1,n+\Delta(G)}).
\]
That is, we can compute the degree vector in terms of the (last non-zero) Betti numbers.
\end{proposition}

\begin{proof}
Let $v$ be as in Proposition \ref{niceidea}, i.e., let $v_i=\binom n{i+1}$. We notice that the ``effective length'' of the vector $v$ is $n-1$, while that of $A\mathbf{d}_{G}$ is $n+\Delta(G)-1$. Therefore the Betti number $\beta_{i,i+1}$ is equal to the entry $(A\mathbf{d}_{G})_i$, for every $i=n,\dots,n+\Delta(G)-1$, whereas $\beta_{n-1,n}=(A\mathbf{d}_{G-v})_{n-1}=(A\mathbf{d}_{G})_{n-1}-1$. The entries of the matrix $A$ that we use in these computations are just
$$A_{ij},\qquad\text{for}\quad i=n-1,\dots,n+\Delta(G)-1,\quad j=1,\dots,\Delta(G)+1,$$
and therefore we define the square matrix $B$ as in the statement above. More explicitly, this submatrix of $A$ has the form
$$B=\left(\begin{array}{ccccc}
1 & B_{12} & B_{13} & \dots & B_{1,\Delta(G)+1} \\
0 & 1 & B_{23} & \dots & B_{2,\Delta(G)+1}\\
0 & 0 & 1 & \dots & B_{3,\Delta(G)+1}\\
\vdots & \vdots & \ddots & \ddots & \vdots\\
0 & \dots &\dots&0&1
\end{array}\right).$$
This matrix is invertible and so we  have the stated formula. 
It only remains to prove that the matrix $C$ with entries $C_{ij}:=(-1)^{i+j}B_{ij}$ is the inverse of $B$. Next we check that $BC$ is the identity matrix: it is clear that $(BC)_{ij}=0$ if $j<i$, and also in the sum
$$(BC)_{ii}=(-1)^i\sum_{k=1}^{\Delta(G)+1}(-1)^k\binom{k+n-2}{i+n-2}\binom{i+n-2}{k+n-2}$$ the only non-zero summand corresponds to $k=i$ and it is $1$. So let $i<j$ and consider
\begin{align*}
(BC)_{ij}&=(-1)^j\sum_{k=i}^j(-1)^k\binom{k+n-2}{i+n-2}\binom{j+n-2}{k+n-2}\\
&=(-1)^{i+j}\frac{(j+n-2)!}{(i+n-2)!}\sum_{k=0}^{j-i}(-1)^k\frac1{k!(j-i-k)!}.
\end{align*}
This is zero because $\sum_{k=0}^N(-1)^k\frac1{k!(N-k)!}=\frac1{N!}\sum_{k=0}^N(-1)^k\binom Nk$ is zero for all $N>0$ (see  Lemma~\ref{lem:thesumiszero}).
\end{proof}

\subsection{From degree vector to Boij--S\"oderberg coefficients}\label{fdvtobsc}

For a graph $G$, by \emph{Boij--S\"oderberg coefficients of $G$} we mean the Boij--S\"oderberg coefficients of $S/I_G$, where $I_G$ is the edge ideal of $G$.

\begin{theorem}\label{thm:main}
Let $G$ be a graph with $n$ vertices and $m\ge n$ edges, and let $\mathbf{d}_{G}=(d_0,d_1,\ldots,d_{n-1})$ be its  degree vector.
Then the $j$th Boij--S\"oderberg coefficient of $BL(G)$ is 
$$c_j=\left\lbrace\begin{array}{ll}
	0 &\text{if } j\leq n-2,\\
	&\\
	\frac{d_0}{j}+\frac{\sum_{i=1}^{n-1}d_i}{j(j+1)}-\frac{n}{j(j+1)}=\frac{d_0}n&\text{if } j=n-1,\\
	&\\
	\frac{d_{j-n+1}}{j}+\frac{\sum_{i=j-n}^{n-1}d_i}{j(j+1)} &\text{if } n-1<j\leq 2n-2,\\
	&\\
	0 & \text{if }j>2n-2.
\end{array}\right.$$
\end{theorem}

\begin{proof}
Putting Lemma~\ref{lem:BSfromES} and~\eqref{bellapensata} together, we find that the $j$th Boij--S\"oderberg coefficient of $BL(G)$ is
\begin{align*}c_j&=(-1)^j\sum_{k=1}^nd_{k-1}\sum_{i=1}^{n+m-1}(-1)^i\frac1i\binom{i+1}{j+1}\binom{k+n-2}i\\
&\qquad+(-1)^{j+1}\sum_{i=1}^{n+m-1}(-1)^i\frac1i\binom{i+1}{j+1}\binom{n}{i+1}\\
&=
\left\lbrace\begin{array}{ll}
(-1)^j\sum_{k=1}^nd_{k-1}\sum_{i=1}^{n+m-1}(-1)^i\frac1i\binom{i+1}{j+1}\binom{k+n-2}i-\frac n{j(j+1)} & \text{if }j\le n-1,\\
&\\
(-1)^j\sum_{k=1}^nd_{k-1}\sum_{i=1}^{n+m-1}(-1)^i\frac1i\binom{i+1}{j+1}\binom{k+n-2}i & \text{if $j>n-1$.}
\end{array}\right.
\end{align*}

For convenience of notation, write $b_{j,k}:=\sum_{i=1}^{n+m-1}(-1)^i\frac1i\binom{i+1}{j+1}\binom{k+n-2}i$.
Clearly for $j>2n-2$ we have that $b_{j,k}=0$ as every summand is zero. 
Now fix $k$ and suppose $j\leq 2n-2$.  In order for $b_{j,k}$ to be non-zero, $j\le k+n-2$ must be satisfied. In the boundary condition $j=k+n-2$ we can easily compute the sum, as there is only the value $i=j$ for which both the binomial coefficients are non-zero. Hence the sum is
$$b_{j,j-n+2}=(-1)^j/j.$$
Finally, if $j<k+n-2$, then
\begin{align*}
b_{j,k}&=\sum_{i=1}^{n+m-1}(-1)^i\frac1i\binom{i+1}{j+1}\binom{k+n-2}i\\
&=\frac{1}{j(j+1)}\sum_{i=1}^{n+m-1}(-1)^i(i+1)\binom{i-1}{j-1}\binom{k+n-2}i.
\end{align*}
Consider now the polynomial $P(x):=(x+1)\binom{x-1}{j-1}$, which has degree $j$. If we set $N:=k+n-2$, by Lemma~\ref{lem:thesumiszero} we have that
\begin{align*}
b_{j,k}&=\frac{1}{j(j+1)}\sum_{i=1}^{n+m-1}(-1)^i(i+1)\binom{i-1}{j-1}\binom{k+n-2}i\\
&=\frac{1}{j(j+1)}\bigg[\sum_{i=0}^{n+m-1}(-1)^i(i+1)\binom{i-1}{j-1}\binom{k+n-2}i-(-1)^{j+1}\bigg]\\
&=\frac{(-1)^j}{j(j+1)}.
\end{align*}
Putting these things together, we find the stated expressions for the Boij--S\"oderberg coefficients.
\end{proof}

\begin{example}\label{examplebs}
Consider again the graph $G$ in Example~\ref{examplebn}. Recall that
$$\begin{array}{rcl}
\pi_7 & = & (28 , 112 , 210, 224, 140, 48, 7, 0, 0, 0, 0, 0, 0, 0 ),  \\
\pi_8 & =  & (36, 168, 378, 504, 420, 216, 63, 8, 0, 0, 0, 0, 0, 0 ), \\
\pi_9 & =  & (45, 240, 630, 1008, 1050, 720, 315, 80, 9, 0, 0, 0, 0, 0 ). 
\end{array}$$
With Theorem~\ref{thm:main}, we can easily find that the Betti vector for $BL(G)$, computed in Example~\ref{examplebn}, can  be encoded as
$$ ( 37 , 177 , 413 , 581 , 525 , 307 , 112 , 23 , 2 , 0 , 0 , 0 , 0, 0) =
\frac{1}{8} \pi_7 + \frac{47}{72}\pi_8 + \frac{2}{9}\pi_9.$$
Or to be more concise, the only non-zero  Boij--S\"oderberg coefficients are
\begin{align*} 
c_7 &= \frac{d_1}{7} + \frac{d_2+d_3+d_4+d_5+d_6}{7\cdot 8}=\frac{1}{8},\\
c_8 &= \frac{d_2}{8} + \frac{d_3+d_4+d_5+d_6}{8\cdot 9} =\frac{47}{72},\\
c_9 &= \frac{d_3}{9} + \frac{d_4+d_5+d_6}{9\cdot 10} = \frac{2}{9}
\end{align*}
where $d_i$ is the number of vertices of degree $i$ in $G$.
\end{example}

\subsection{From degree vector to anti-lecture hall compositions}\label{sec:antilec}

For a brief introduction to anti-lecture hall compositions, see Section~\ref{sec:antilecintro}.
For an illustration of the next result, see Example~\ref{antilectureexample}.

\begin{proposition}\label{fromdegreetoantilec}
Let $G$ be a graph with $n$ vertices and $m$ edges, and assume that $m\ge n-1$. Denote by $d_k$ the number of vertices of degree $k$ in $G$ and denote by $\lambda$ the anti-lecture hall composition associated to $BL(G)$. Then we have
$$\lambda_j=
\begin{cases}
j & \text{for }j=1,\dots,n,\\
d_{n-1}+d_{n-2}+\dots+d_{j-n+1}& \text{for }j=n,\dots,2n-2,\\
0 & \text{for }j>2n-2.
\end{cases}$$
Note that in particular for $j=n$ we get
$\lambda_n=d_{n-1}+d_{n-2}+\dots+d_0=n$. 
\end{proposition}

\begin{proof}
By applying Lemma~\ref{ESantilecfrombetti} to $BL(G)$, we have $\lambda=\omega(BL(G))\Psi^{-1}$, where $\Psi$ is the invertible square matrix of size $n+m-1$ with entries $\Psi_{ij}=\binom{i-1}{j-1}$ and $\lambda=(\lambda_1,\dots,\lambda_{n+m-1})$. Using~\eqref{bellapensata} and denoting $\Xi:=(\Psi^{-1})^T$, we find
\begin{equation}\label{vacca}
\lambda^T=(\Psi^{-1})^TA\mathbf{d}_{G}-(\Psi^{-1})^Tv=\Xi A\mathbf{d}_{G}-\Xi v.
\end{equation}
Thus we only need to see explicitly what $\Xi A$ and $\Xi v$ are. We claim that $\Xi A$ is of the following form (see also Example~\ref{antilectureexample}): the entry $(\Xi A)_{jk}$ is $1$ for $j\le k+n-2$ and it is $0$ for $j>k+n-2$. To prove this, let $j\le k+n-2$ and note that
\begin{align*}
(\Xi A)_{jk}&=\sum_{i=1}^{n+m-1}\Xi_{ji}A_{ik}=\sum_{i=1}^{n+m-1}(-1)^{j+i}\binom{i-1}{j-1}\binom{k+n-2}i\\
&=(-1)^j\sum_{i=j}^{k+n-2}(-1)^i\binom{i-1}{j-1}\binom{k+n-2}i\\
&=(-1)^j\bigg[\sum_{i=0}^{k+n-2}(-1)^iP(i)\binom{k+n-2}i-P(0)\bigg]\\
&=(-1)^{j-1}P(0)=1,
\end{align*}
where we apply Lemma~\ref{lem:thesumiszero} by considering the polynomial
$$P(x):=\binom{x-1}{j-1}=\frac{(x-1)(x-2)\dots(x-j+1)}{(j-1)!}.$$ Moreover, for $j>k+n-2$ every term in the sum expressing the entry $(\Xi A)_{jk}$ vanishes, and then $(\Xi A)_{jk}=0$. 

For what concerns $\Xi v$, in order to obtain the formulas for the $\lambda_j$'s starting from~\eqref{vacca}, we only need to prove that 
$$(\Xi v)_j=\begin{cases}
n-j &\text{if }j< n\\
0 &\text{if }j\ge n.
\end{cases}$$ This holds since for $j<n$ we have
\begin{align*}
(\Xi v)_j &= \sum_{i=1}^{n+m-1}\Xi_{ji} v_i=\sum_{i=1}^{n+m-1}(-1)^{j+i}\binom{i-1}{j-1}\binom n{i+1}\\
&=(-1)^j\sum_{i=j}^{n-1}(-1)^i\binom{i-1}{j-1}\binom n{i+1}\\
&=(-1)^{j+1}\sum_{i=j+1}^n(-1)^i\binom{i-2}{j-1}\binom ni\\
&=(-1)^{j+1}\bigg[\sum_{i=0}^n(-1)^iP(i)\binom ni-P(0)+P(1)n\bigg]\\
&=(-1)^{j+1}\big[-(-1)^{j+1}j+(-1)^{j+1}n\big]\\
&=n-j,
\end{align*}where the polynomial we choose this time, in order to apply Lemma~\ref{lem:thesumiszero}, is
$$Q(x):=\binom{x-2}{j-1}=\frac{(x-2)(x-3)\dots(x-j)}{(j-1)!},$$
and for $j\ge n$ we have $(\Xi w)_j=0$ as every term in the big sum has as a multiplying factor a vanishing binomial coefficient.
\end{proof}

\begin{example}\label{antilectureexample}
Consider again the graph $G$ from Example~\ref{examplebn} and Example~\ref{examplebs}. By Proposition~\ref{fromdegreetoantilec}, we find that
\[ \begin{array}{rcl}
\lambda_1 & = & 1 \\
\lambda_2 & = & 2 \\
\lambda_3 & = & 3 \\
\lambda_4 & = & 4 \\
\lambda_5 & = & 5 \\
\lambda_6 & = & 6 \\
\lambda_7 & = & 7 \\
\lambda_8 & = &  d_3 + d_2=7   \\
\lambda_9 & = &  d_3 =2   \\
\lambda_{10}=\lambda_{11}=\lambda_{12} & = & 0.
\end{array} \]
Moreover, the matrix and vector featured in the proof of Proposition~\ref{fromdegreetoantilec} have the following forms:
$$\Xi A=\left(\begin{array}{ccccccc}
1&1&1&1&1&1&1\\
1&1&1&1&1&1&1\\
1&1&1&1&1&1&1\\
1&1&1&1&1&1&1\\
1&1&1&1&1&1&1\\
1&1&1&1&1&1&1\\
0&1&1&1&1&1&1\\
0&0&1&1&1&1&1\\
0&0&0&1&1&1&1\\
0&0&0&0&1&1&1\\
0&0&0&0&0&1&1\\
0&0&0&0&0&0&1\\
0&0&0&0&0&0&0\\
0&0&0&0&0&0&0\\
\end{array}\right),\qquad
\Xi v=\left(\begin{array}{c}
6\\
5\\
4\\
3\\
2\\
1\\
0\\
0\\
0\\
0\\
0\\
0\\
0\\
0\\
\end{array}\right).$$
\end{example}

\section{About the complement $\overline{BL(G)}$}\label{sec:mainComp}

\begin{proposition}
Let $G$ be a graph with $n$ vertices and $m$ edges. Then, for every integer $j\ge1$ we have
$$\beta_{j,j+1}\big(\overline{BL(G)}\big)=m\binom{m+n-3}j-\binom m{j+1}.$$
\end{proposition}

\begin{proof}
We prove the claim in two steps:
\begin{enumerate}
\item[(1)] We generalize the Booth--Lueker construction to multi-graphs and prove the following: let $G$ be a multi-graph with, in particular, vertices $u,v,w$ and an edge $uv$, and let $G'$ be the same multi-graph except for the fact that we remove that edge $uv$ and add an edge $uw$; then the Betti numbers of $\overline{BL(G)}$ and $\overline{BL(G')}$ are equal.
\item[(2)] We prove the stated formula for a particular multi-graph $H$: the one with all the $m$ edges between two fixed vertices, and $n-2$ other isolated vertices. Due to the fact that we can reach $H$ starting from any $G$ with $n$ vertices and $m$ edges iterating the transformation in the first step, and since the Betti numbers stay constant at each iteration, we will have the formula for any $G$.
\end{enumerate}

The generalization of the Booth--Lueker construction to multi-graphs is done in the natural intuitive way: $BL(G)$ will be a simple graph with a ``left part'' with the original vertices of $G$, all connected to each other with one edge, and a ``right part'' with as many vertices as the edges of $G$, connected to the respective ends in the left part.

Let us prove part $(1)$, by using the formula in Theorem~\ref{bettibyconncomp}. Let $e$ be the edge between $u$ and $v$ that we ``move'' between $u$ and $w$. The vertex sets of $BL(G)$ and $BL(G')$ can be written in the same way, the only difference is that the vertex $e$ in the right part of $BL(G)$ is connected to $u$ and $v$ on the left, whereas $e$ in the right part of $BL(G')$ is connected to $u$ and $v$.

Let $S$ be a subset of the vertices of $BL(G)$ and let's see how the number of connected components of $BL(G)[S]$ and $BL(G')[S]$ differ.
\begin{enumerate}
\item[(i)] If $e\notin S$, then $BL(G)[S]=BL(G')[S]$.
\item[(ii)] If $u\in S$, then $e$ is path-connected to $v$ and $w$ if they are in the graph, whether the ``moving edge'' is there or not, so the numbers of connected component of $BL(G)[S]$ and $BL(G')[S]$ are the same.
\item[(iii)]  Let $e\in S$ and $u\notin S$. Then we have three possibilities:
\begin{enumerate}
\item[(a)] If both $v$ and $w$ are in $S$, then the number of connected components of $BL(G)[S]$ and $BL(G')[S]$ are the same.
\item[(b)] If $v\in S$ and $w\notin S$, then the number of connected components of $BL(G)[S]$ increases by one after the move to $BL(G')[S]$, as $e$ becomes a new isolated vertex.
\item[(c)] If $v\notin S$ and $w\in S$, then the number of connected components of $BL(G)[S]$ decreases by one after the move to $BL(G')[S]$, as $e$ is no longer a new isolated vertex.
\end{enumerate}
\end{enumerate}
We have found the following: only if $e\in S$, $u\notin S$ and exactly one of $v$ and $w$ is in $S$ we have a change in the number of connected components. Consider the map
\begin{align*}
\varphi\colon\{S\subseteq V(BL(G))&\mid v\in S, w\notin S, e\in S, u\notin S\}\to\\
&\{S\subseteq V(BL(G))\mid v\notin S,w\in S,e\in S,u\notin S\}
\end{align*}
defined by $S\mapsto (S\cup\{w\})\setminus\{v\}$. Then $\#S=\#\varphi(S)$, and $\varphi$ is a bijection between those subsets there the number of connected components increases by one and those where it decreases by one. So in total the graded Betti numbers stay constant.

\begin{figure}
\begin{center}
\includegraphics[width=4cm
]{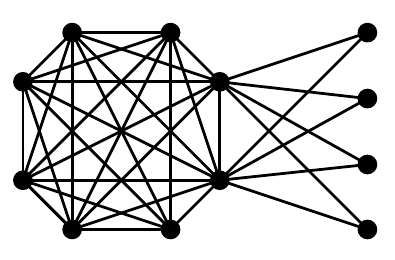}
\caption{The Booth--Lueker graph looking like a pineapple.}
\label{fig:pineapple}
\end{center}
\end{figure}

Let us now prove the second part. First, let us denote $H$ the multi-graph on $n$ vertices and $m$ edges, where all the edges are between two fixed vertices $u$ and $v$. For instance, if $n=8$ and $m=4$, then the Booth--Lueker graph of $H$ is depicted in Figure~\ref{fig:pineapple} and therefore we will refer to $BL(H)$ as the \emph{pineapple}. 

For each integer $j\ge1$, let us compute $\beta_{j,j+1}\big(\overline{BL(H)}\big)$ by using the formula in Theorem~\ref{bettibyconncomp}. For any subset $W$ of $j+1$ vertices in $BL(H)$, the number of connected components of $BL(H)[S]$ can vary from $1$ to $j+1$. We are not interested in the cases with only one connected component, as they give zero in the sum. If $i+1$ is the number of connected component, with $1\le i\le j-1$, then we have $\binom{n-2}{j+1-i}\binom mi$ choices for such a $W$, by selecting $i$ vertices in the right part of the pineapple and $j+1-i$ in the left part, exept for $u$ and $v$. For $j+1$ connected components, we can either choose $j+1$ vertices in the right part of the pineapple, or one vertex among $n-2$ vertices on the left and $j$ on the right. Therefore, we find
\begin{align*}
\beta_{j,j+1}\big(\overline{BL(H)}\big)&=\sum_{i=1}^{j-1}i\binom{n-2}{j+1-i}\binom mi+j\bigg[\binom m{j+1}+(n-2)\binom mj\bigg]\\
&=m\sum_{i=1}^{j-1}\binom{n-2}{j+1-i}\binom{m-1}{i-1}+j\bigg[\binom m{j+1}+(n-2)\binom mj\bigg]\\
&=m\bigg[\sum_{i=0}^j\binom{n-2}{j-1}\binom{m-1}i-(n-2)\binom{m-1}{j-1}-\binom{m-1}j\bigg]\\
&\!\qquad\qquad\qquad\qquad\qquad\qquad\qquad+j\bigg[\binom m{j+1}+(n-2)\binom mj\bigg]\\
&\stackrel{(*)}=m\binom{m+n-3}j-m\binom{m-1}j+j\binom m{j+1}\\
&=m\binom{m+n-3}j-\binom m{j+1},
\end{align*}
where for $(*)$ we used that $\sum_{i=0}^j\binom{n-2}{j-1}\binom{m-1}i=\binom{m+n-3}j$ (see Lemma~\ref{littlebinomialformulas}) and that $m\binom{m-1}{j-1}=j\binom mj$.
\end{proof}

\begin{theorem}\label{thm:mainComp}
Let $G$ be a graph with $n$ vertices and $m$ edges. Then the $i$-th Boij--S\"oderberg coefficient of $\overline{BL(G)}$ is
$$c_i=\left\lbrace\begin{array}{ll}
	0 &\text{if } i< m,\\
	&\\
	\frac m{(i+1)i}&\text{if } m\le i\le m+n-4,\\
	&\\
	\frac mi &\text{if } i=m+n-3,\\
	&\\
	0 & \text{if } i>m+n-3.
\end{array}\right.$$
\end{theorem}

\begin{proof}
We use again the formula in Lemma~\ref{lem:BSfromES}. We find
\begin{align*}
c_i&=\sum_{j=1}^{n+m-1}\beta_{j,j+1}\big(\overline{BL(G)}\big)(-1)^{j-i}\frac1j\binom{j+1}{i+1}\\
&=m\sum_{j=1}^{n+m-1}(-1)^{j-i}\frac1j\binom{m+n-3}j\binom{j+1}{i+1}-\sum_{j=1}^{n+m-1}(-1)^{j-i}\frac1j\binom m{j+1}\binom{j+1}{i+1}\\
&=-\frac{m(m+n-3)}{(i+1)i}(-1)^{i+m+n}\sum_{j=0}^{n+m-4}(-1)^{m+n-4-j}\binom{m+n-4}j\binom j{i-1}\\
&\qquad+\frac m{(i+1)i}(-1)^{i+m+n+1}\sum_{j=0}^{n+m-4}(-1)^{m+n-3-j}\binom{m+n-3}j\binom{j-1}{i-1}\\
&\qquad-\frac m{(i+1)i}(-1)^{i+m+1}\sum_{j=0}^{m-1}(-1)^{m-1-j}\binom{m-1}j\binom{j-1}{i-1}-\frac m{(i+1)i},
\end{align*}
from which we conclude, by Lemma~\ref{littlebinomialformulas}.
\end{proof}

\begin{proposition}
Let $G$ be a simple graph with $n$ vertices and $m$ edges. The associated anti-lecture hall composition $\lambda=(\lambda_1,\dots,\lambda_{n+m-1})$ is such that
$$\lambda_j=
\begin{cases}
j &\text{if }j\le m,\\
m&\text{if }m<j\le m+n-3,\\
0&\text{if }j>m+n-3.
\end{cases}$$
\end{proposition}
\begin{proof}
If we apply Lemma~\ref{ESantilecfrombetti} to $\overline{BL(G)}$, since the corresponding matrix $\Psi^{-1}$ has $ij$-entry equal to $(-1)^{i+j}\binom{i-1}{j-1}$, we find that
\begin{align*}
\lambda_j &=\sum_{i=1}^{n+m-1}\beta_{i,i+1}\big(\overline{BL(G)}\big)\Psi_{ij}^{-1}\\
&=m\sum_{i=1}^{n+m-1}(-1)^{i+j}\binom{m+n-3}i\binom{i-1}{j-1}-\sum_{i=1}^{n+m-1}(-1)^{i+j}\binom m{i+1}\binom{i-1}{j-1}\\
&=(-1)^{j+n+m+1}m\sum_{i=0}^{n+m-1}(-1)^{m+n-3-i}\binom{m+n-3}i\binom{i-1}{j-1}\\
&\qquad+(-1)^{j+m}\sum_{i=0}^{n+m}(-1)^{m-i}\binom{i-2}{j-1}-(-1)^j\binom{-2}{j-1}\\
&=(-1)^{j+m+n+1}m\binom{-1}{j-m-n+2}+(-1)^{j+m}\binom{-2}{j-m-1}-(-1)^j\binom{-2}{j-1},
\end{align*}
where we applied Lemma~\ref{littlebinomialformulas}. Now, depending on $j$, we find the stated expressions for $\lambda_j$.
\end{proof}

\end{document}